\documentclass[12pt,leqno]{amsart}

\usepackage{amsmath,amsfonts,amssymb,amsthm,mathrsfs}
\usepackage{epic}  
\usepackage[all]{xy}   
\usepackage{graphicx}
\usepackage{epsfig}

\usepackage[usenames]{color}

\newtheorem{thm}{Theorem}[section]

\newtheorem{cor}[thm]{Corollary}
\newtheorem{lem}[thm]{Lemma}

\theoremstyle{definition}

\newtheorem{rem}[thm]{Remark}
\newtheorem{defn}[thm]{Definition}

      \newcommand{\N}{{\mathbb N}}
      \newcommand{\R}{{\mathbb R}}
\newcommand{\Q}{{\mathbb Q}}

\newcommand{\supp}{\operatorname{supp}}

             \newcommand{\Lip}{\operatorname{Lip}}
              \newcommand{\lip}{\operatorname{lip}}

\newcommand{\ap}{\text{ap}}

\newcommand{\veps}{\varepsilon}

\newcommand{\jint}{\int\hspace{-4.15mm}\frac{\,\,\,}{}}

\newcommand{\Leb}[1]{{\mathscr L}^{#1}}

\def\Xint#1{\mathchoice
{\XXint\displaystyle\textstyle{#1}}%
{\XXint\textstyle\scriptstyle{#1}}%
{\XXint\scriptstyle\scriptscriptstyle{#1}}%
{\XXint\scriptscriptstyle\scriptscriptstyle{#1}}%
\!\int}
\def\XXint#1#2#3{{\setbox0=\hbox{$#1{#2#3}{\int}$}
\vcenter{\hbox{$#2#3$}}\kern-.5\wd0}}

\def\dashint{\Xint-}

\begin{document}
\title[On characterization of
approximate differentiability on metric spaces]{On Whitney-type characterization
of \\ approximate differentiability \\
on metric measure spaces}
\author{E. Durand-Cartagena, L. Ihnatsyeva, R. Korte, M. Szuma\'nska}
\address{E. D.-C.: Departamento de Matemática Aplicada. ETSI Industriales, UNED c/Juan del Rosal 12
Ciudad Universitaria, 28040 Madrid, Spain}
\email{edurand@ind.uned.es}
\address{L. I.: Department of Mathematics and Statistics, P.O. Box 68, FI-00014 University of Helsinki, Finland}
\email{lizaveta.ihnatsyeva@helsinki.fi}
\address{R. K.: Department of Mathematics and Statistics, P.O. Box 68, FI-00014 University of Helsinki, Finland}
\email{riikka.korte@helsinki.fi}
\address{M. S.: Faculty of Mathematics, Informatics, and Mechanics University of Warsaw, Banacha 2, 02-097 Warszawa, Poland}
\email{m.szumanska@mimuw.edu.pl}
\date{}
\keywords{approximate differentiability, metric space, strong
measurable differentiable structure, Whitney theorem}
\subjclass[2010]{26B05, 28A15, 28A75, 46E35}
\maketitle

\begin {abstract}
We study approximately differentiable functions on metric measure spaces admitting a Cheeger differentiable structure. The main result is a Whitney-type characterization of approximately differentiable functions in this setting. 
As an application, we prove a Stepanov-type theorem and consider approximate differentiability of Sobolev, $BV$ and maximal functions.
\end {abstract}

\bigskip
A classical theorem of Luzin states that a measurable function which is finite almost everywhere coincides with a continuous function outside a set of arbitrary small measure.
A function with such a property is said to satisfy the Luzin property of order zero. The reverse implication in Luzin theorem also holds true and thus the
Luzin property actually characterizes measurable functions. By the aid of Lebesgue differentiation theorem, one can see
that a function defined on $\R^n$ has the Luzin property of order zero if and only if it is approximately continuous almost everywhere.
This characterization is known as Denjoy-Luzin theorem, see \cite{De,Lu}.

For more regular functions, it is natural to expect Luzin properties of higher order.
Indeed, Whitney \cite{Wh} proved that approximately differentiable functions are precisely the functions that have the Luzin property
of order one, in the sense that they are smooth on ``nearly'' all of their domain.

The concept of approximate continuity makes perfect sense for functions defined on arbitrary metric measure spaces.
The same reasoning as in the Euclidean case shows that Denjoy-Luzin theorem holds true for metric spaces equipped with a doubling measure,
see Theorem \ref{zeroLuzin}.
Our aim is to extend the Whitney theorem to a more general setting.
Recently, there has been an intensive research, where a first order differential calculus has been developed on metric measure spaces.
For a general introduction to the subject, we mention here the survey works by Heinonen \cite{He,He2}, Heinonen-Koskela \cite{HeKo}, Ambrosio-Tilli \cite{AT},
Haj{\l}asz-Koskela \cite{HKo}, Semmes \cite{Se} and Bj\"orn-Bj\"orn \cite{BB}.
The standard assumptions, which allow the first order differential calculus, include that
the measure is doubling and that the space supports a $p$-Poincar\'e inequality.

Cheeger \cite{Ch}  constructed a measurable differentiable structure for the above mentioned class of metric spaces (see also Keith \cite{Ke1})
in such a way that Lipschitz functions can be differentiated almost everywhere with respect to this differentiable structure (Rademacher's theorem).
Cheeger's differentiable structure provides a means to study approximate differentiability in metric measure spaces.
The concept of approximate differentiability in this setting has been already considered by Keith \cite{Ke2} and by Bate and Speight \cite{BS}. See also Basalaev and Vodopyanov \cite{BV} for the study of approximate differentiability and Whitney-type theorems in the sub-Riemannian setting.

In this work, we consider the class of approximately differentiable functions in spaces that admit a Cheeger differentiable structure.
The main result of this paper, Theorem \ref{main}, is a Whitney-type characterization of approximate differentiability in the metric setting.
The Whitney theorem has its own interest as a classical result of real analysis, but the characterization of Luzin property has also been used,
for example, to prove regularity properties of different function spaces, especially when differentiability is not always guaranteed.
This is the case for example for Sobolev and $BV$ functions.
For approximate differentiability properties of Sobolev and $BV$ functions in the Euclidean case, one can consult \cite{EG}.

We apply our main result in three different directions.
The first one is related to Stepanov theorem, and the other two are in connection with differentiability properties of
Sobolev functions and the Hardy-Littlewood maximal function.


The Stepanov theorem \cite{St} states that a function is differentiable precisely on the set of points where a certain local growth condition holds.
We prove  an approximate version of Stepanov theorem (Theorem \ref{StepanovEqivalLuzin})
in the metric setting and use it to give another characterization of approximate differentiability (Corollary \ref{StepanovTypeCharact})
and an alternative proof for Stepanov theorem proved by Balogh-Rogovin-Z\"urcher  \cite{BRZ}.
Our methods follow the lines of proof of the classical approximate Stepanov theorem in \cite{F}.

The obtained characterizations allow us to give a simple proof of the approximate differentiability for Sobolev functions (Haj{\l}asz-Sobolev spaces
\cite{H} and Newtonian spaces \cite{Sh}) and $BV$ functions (Miranda \cite{M}) in the metric setting, see Corollary \ref{sobolev}.
Notice that approximate differentiability properties can be also deduced from existing results. See Bj\"{o}rn \cite{Bj} and Ranjbar-Motlagh \cite{R}.

To finish, we use the Whitney type characterization to show in Theorem \ref{maxdiff} that the notion of approximate
differentiability in metric spaces is preserved under the action of the discrete maximal operator.
The analogous statement for the regular Hardy-Littlewood maximal operator in the Euclidean setting was proved earlier by Haj{\l}asz-Mal\'y  in \cite{HM}.  Buckley \cite{Bu} has shown that for a metric space with a doubling measure,
the maximal operator may not preserve Lipschitz and H\"older spaces. Therefore some Lipschitz-type estimates that were used to prove the approximate continuity in Euclidean spaces do not hold in more general spaces. In order to have a maximal function which preserves, for example, the Sobolev spaces on metric spaces, Kinnunen and Latvala \cite{KL}
used the discrete maximal function.
Notice that in many applications Hardy-Littlewood maximal operator can be replaced by discrete maximal operator as they are comparable by two-sided estimates \cite{KL} . 

Luzin properties of order $k$ for $k>1$ have been studied by Bojarski \cite{Bo}, Liu \cite{Li}, Liu-Tai \cite{LT1}, \cite{LT2} in the Euclidean setting.
See also \cite{EG}. In this paper, we only consider Luzin properties of order $1$, since the theory for higher order derivatives has not been developed yet in the metric setting. However, it would be interesting to extend these results to higher order cases at least for lower dimensional subsets of $\R^{n}$.

The paper is organized as follows. In Section \ref{preli}, we first briefly recall the concepts of approximate continuity
and approximate differentiability in the Euclidean setting. After that we give some standard notation and relevant notions regarding metric
spaces supporting a doubling measure that enable us to define approximate differentiability in this more general context.
Section \ref{Characterization}  contains the main result of this paper: a Whitney-type characterization of approximately differentiable
functions in this setting as well as an Stepanov-type characterization.
In Section \ref{differentiability}, we use the obtained characterizations to show the approximate differentiability for Sobolev and $BV$ functions.
In the final Section \ref{maxim}, we prove that approximate differentiability a.e. is preserved under the action of the discrete maximal operator.

\section{Preliminaries}\label{preli}

\subsection{Approximate differentiability in $\R^n$}

We say that $l\in \mathbb{R}$ is the \em approximate limit \em of a function $f:\R^n\to\R$ as $y\rightarrow x$,
and write
$$
\ap\lim_{y\rightarrow x} f(y)=l,
$$
if for every $\varepsilon>0$,
$x\in\R^n$ is a density point for the set $\{y:|f(y)-l|)<\varepsilon\}$.

Observe that equivalently we can formulate the definition in the following way. There exists $A\subset \R^n$ with $x$ a point of density for $A$ such that
\[
\lim_{\substack{y\to x\\ y\in A}} |f(y)-l|=0.
\]
If the approximate limit $l$ exists and $f(x)=l$, then we say that $f$ is {\em approximately continuous} at $x$.

Using the notion of approximate limit one can define the approximate differential.
\begin{defn}
Let $E\subset\R^n$ and $f:E\to\R$.  We say that $f$ is \em approximately differentiable \em at $x\in E$ if there exists a vector
$L=(L_1,\cdots,L_n)$ such that
$$
\ap\lim_{y\rightarrow x}\frac{|f(y)-f(x)-L\cdot(y-x)|}{|y-x|}=0.
$$
\end{defn}
Approximate differentiability is a much weaker notion than
differentiability. The function $f:[0,1]\to\R$, $f(x)=1$ if
$x\in\R\setminus\Q$ and $f(x)=0$ if $x\in\Q$ is approximately
differentiable almost everywhere but nowhere differentiable. On
the other hand, even a continuous function might be approximately
differentiable almost nowhere \cite[p.297]{S}.

The following characterization of approximate differentiability was given by Whitney in \cite{Wh}. See also \cite[Theorem 3.1.8]{F}.

\begin{thm}{\bf\cite{Wh}}\label{Whitney}
Let $E\subset\R^n$ and $f:E\to\R$ be a $\Leb{n}$-measurable function. Then the following conditions are equivalent:
\begin{itemize}
\item[$\text{\em(a)}$] $f$ is approximately differentiable $\Leb{n}$-a.e.
\item[$\text{\em(b)}$]$f$ has a Lipschitz Luzin approximation, that is,
for any $\veps>0$ there is a closed set $F\subset E$ and a locally Lipschitz function $g:\R^n\to\R$ such that $f_{|F}=g_{|F}$ and $\Leb{n}(E\setminus F)<\veps$.
\item[$\text{\em(c)}$] $f$ has a smooth Luzin approximation, that is,
for any $\veps>0$ there is a closed set $F\subset E$ and a function $g\in {\mathcal C}^1(\R^n)$ 
such that $f_{|F}=g_{|F}$ and $\Leb{n}(E\setminus F)<\veps$.
\item[$\text{\em(d)}$] $f$ induces the following decomposition
$$
E=\bigcup_{i=1}^{\infty}E_i\cup Z,
$$
where $E_i$ are disjoint closed sets, $f_{|E_i}$ is Lipschitz continuous and $Z$ has measure zero.
\end{itemize}
\end{thm}

\subsection{Approximate differentiability in metric measure spaces}

Our main aim is to extend the statement of Whitney's theorem (Theorem \ref{Whitney}) to the more general setting of a metric measure space.
To formulate a definition of approximate differentiability in such setting, we employ the ideas of Cheeger
\cite{Ch}, who extended the fundamental notions of first order differential calculus to a general class of metric spaces.
We start with several standard definitions.

Throughout the paper $(X,d,\mu)$ refers to a \em metric measure space\em, where $(X,d)$ is a \em separable \em metric space and $\mu$ is \em a non-trivial, locally finite  Borel regular \em measure.

 For $x\in X$ and $r>0$ we denote by $B(x,r):=\{y\in X: d(x,y)<r\}$ 
the open ball of radius $r$ centered at $x$.

\noindent
One of the natural assumptions posed on the measure is the doubling condition.

\begin{defn}
A measure $\mu$ on $X$ is called  \em doubling \em if there is a
positive constant $C_{\mu}$ such that
$$
\mu(B(x,2r))\leq C_{\mu}\,\mu(B(x,r)),
$$
for each $x\in X$ and $r>0$.
\end{defn}

Recall that a point $x\in X$ is a \em density point \em for a $\mu$-measurable set $A\subset X$, if
$$
\lim_{r\rightarrow 0}\frac{\mu(B(x,r)\cap A)}{\mu(B(x,r))}=1.
$$

The following theorem gives a characterization of approximate continuity in the metric setting
and gives an interpretation of the notion of ``$0$-smoothness''.
For a proof of the theorem in the Euclidean setting see \cite{F} or \cite{EG}. See also \cite{Bo} for a nice discussion about the role of
Luzin-Denjoy theorem.

We state the result without a proof
since it follows the lines of the classical setting.
One just need to have in mind that Lebesgue differentiation theorem holds in spaces equipped with a doubling measures (see \cite[1.8]{He}).

\begin{thm}\label{zeroLuzin}
Let $(X,d,\mu)$ be a metric measure space with $\mu$-doubling. Let $E\subset X$ be a bounded $\mu$-measurable set and $f:X\to\R$. The following conditions are equivalent:
\begin{itemize}
\item[$\text{\em(a)}$] $f$ is $\mu$-measurable.
\item[$\text{\em(b)}$] $f$ is approximately continuous $\mu$-a.e. $x$ in $E$.
\item[$\text{\em(c)}$] $f$ is quasicontinuous, that is, for each $\veps>0$ there is a closed set $F\subset E$ with $\mu(E\setminus F)<\veps$ and $f_{|F}$ is continuous. In other words, $f$ has a Luzin approximation of order zero.

\item[$\text{\em(d)}$]  $f$ induces a (zero order) Luzin decomposition of $E$, that is,
$$
E=\bigcup_{i=1}^{\infty}E_i\cup Z,
$$
where $E_i$ are closed measurable sets such that $f_{|E_i}$ is continuous and $Z$ has measure zero.
\end{itemize}

\end{thm}

The structure of metric spaces endowed with a doubling measure has turned out to be too weak to develop a first order differential calculus involving derivatives and therefore extra conditions are needed. The following Poincar\'e inequality creates a link between the measure, the metric and the upper-gradient,
and is ubiquitous in analysis on metric spaces.  We recall that non-negative Borel function $g$ on $X$ is a \em upper
gradient \em of an extended real-valued function $f$ on $X$ if $
|f(\gamma(a))-f(\gamma(b))|\leq\int_{\gamma}g$
for all \ rectifiable curves $\gamma:[a,b]\to X$.

\begin{defn}
Let $1\leq p\le\infty$. We say that $(X,d,\mu)$ supports a \em weak
$p$-Poincar\'e inequality \em if there are constants $\lambda_p\ge 1$ and $C_p>0$
such that when $f:X\to\R\cup\{-\infty,\infty\}$ is a measurable function, $g:X\rightarrow[0,\infty]$ is an upper gradient of $f$ and $B(x,r)$ is a ball in $X$,
\begin{equation}\label{poinc}
\jint_{B(x,r)}|f-f_{B(x,r)}|\,d\mu\leq C_p\,r\Big(\jint_{B(x,\lambda_p
r)}g^p d\mu\Big)^{1/p}
\end{equation}
if $1\leq p<\infty$, and
$$
\jint_{B(x,r)}|f-f_{B(x,r)}|\,d\mu\leq C_{\infty}\,r\|g\|_{L^{\infty}(B(x,\lambda_{\infty} r))}
$$
if $p=\infty$. 

Here and everywhere below we write
$$
f_{A}=\jint_A f:=\frac{1}{\mu(A)}\int_A f d\mu,
$$
where $A\subset X$ and $0<\mu(A)<\infty$.
\end{defn}

We now recall the theorem of Cheeger \cite{Ch}, which states that a metric space equipped with a doubling measure and having a $p$-Poincar\'e inequality admits a certain differentiable structure
for which Lipschitz functions are differentiable $\mu$-a.e.

\begin{thm}
\label{Cheeger}
Let $X$ be a metric space with a doubling Borel
measure $\mu$,
and suppose
that $X$ supports a weak $p$-Poincar\'e inequality for some $1\leq
p<\infty$. Then there exists a countable collection
$\{(X_{\alpha},\bf {x}_{\alpha})\}_{\alpha\in \Lambda}$ of measurable sets $X_{\alpha}\subset X$ and Lipschitz coordinates
$$
{\bf x}_{\alpha}=(x_{\alpha}^1,\ldots,x_{\alpha}^{N(\alpha)}):
X\longrightarrow\R^{N(\alpha)}
$$
with the following properties:
\begin{enumerate}
\item[$(i)$] $\mu\Big(X\setminus\bigcup_{\alpha} X_{\alpha}\Big)=0$;
\item[$(ii)$] There exists $N\geq 0$ such that $N(\alpha)\leq N$ for each $(X_{\alpha},{\bf x}_{\alpha})$;
\item[$(iii)$] If $f:X\rightarrow \R$ is Lipschitz, then for each $(X_{\alpha},{\bf x}_{\alpha})$ there exists a unique (up to a set of zero measure)
measurable bounded vector valued function $d^{\alpha}f:X_{\alpha}\longrightarrow\R^{N(\alpha)}$ such that
\begin{equation}\label{diff}
\lim_{\substack{y\to x\\ y\neq x}}
\frac{|f(y)-f(x)-d^{\alpha}f(x)\cdot({\bf x}_{\alpha}(y)-{\bf
x}_{\alpha}(x))|}{d(y,x)}= 0
\end{equation}
for $\mu$-a.e. $x\in X_{\alpha}$.
\end{enumerate}
\end{thm}

If a metric measure space $(X,d,\mu)$ satisfies the conclusion of
Theorem \ref{Cheeger}, we say that the space admits a \em
strong measurable differentiable structure\em. In particular, $\{(X_{\alpha},\bf {x}_{\alpha})\}_{\alpha\in \Lambda}$ is said to be a strong measurable differentiable structure for $(X,d,\mu)$.

Notice that although the exponent $p$ is present in the hypothesis of this result,
it has no role in the conclusions. Keith \cite{Ke1} weakened the hypotheses using the $\Lip-\lip$ condition formulated as follows.
There exists a constant $K\geq 1$ such that
$$
\Lip f(x)\leq K\lip f(x)
$$
for all Lipschitz functions $f:X\longrightarrow\R$ and
for $\mu$-almost every $x\in X$, where $\Lip f$ and $\lip f$ denote the upper and lower scaled oscillation functions respectively.
This $\Lip-\lip$ condition is satisfied
by any complete metric space endowed with a doubling measure which admits a $p$-Poincar\'e inequality for some $1\leq
p<\infty$.

See [KlMa] for an accessible introduction to basis of the theory of differentiable structures.

The existence of the differentiable structure allows us to consider the following notion of differentiability of a function.

\begin{defn}
A function $f:X\to\R$ is \em Cheeger differentiable \em at a point
$x\in X_{\alpha}$  with respect to the strong measurable
differentiable structure $\{(X_{\alpha},\bf {x}_{\alpha})\}_{\alpha\in \Lambda}$ if  there
exists a unique vector (\em Cheeger differential \em) $d^\alpha f(x)\in\R^{N(\alpha)}$ such that
\eqref{diff} holds for $f$ at $x$.
\end{defn}

\noindent
Notice that
the definition of the
differentiable structure implies that the uniqueness of the Cheeger
differential can be inferred from its existence almost everywhere on $X$. The
exceptional set depends only on the differentiable structure.






Analogously one can introduce a notion of approximate
differentiability of a function defined on a metric space. See \cite{Ke2} and
\cite{BS}.

\begin{defn} If a metric measure space $(X,d,\mu)$ satisfies the
conclusion of Theorem \ref{Cheeger}, where limit in \eqref{diff}
is replaced with an approximate limit, it is said that the space
admits  \em an
approximate differentiable structure \em
(or \em a measurable differentiable structure\em).
\end{defn}

Recently, Bate and Speight \cite{BS} have proved that if a metric measure
space admits a strong measurable differentiable structure then the measure is
pointwise doubling. They also gave an example showing that if one only requires
an approximate differentiable structure, the measure does not need to be pointwise doubling.

\begin{defn}
\label{DefApDif1} Let $(X,d,\mu)$ be a metric measure space that supports an approximate differentiable
structure $\{(X_{\alpha},\bf {x}_\alpha)\}_{\alpha \in \Lambda}$.
A function $f: X \to \R$ is \em approximately differentiable \em at $x \in
X_\alpha$ with respect to $(X_{\alpha},{\bf x}_\alpha)$  if there
exists a vector $L^\alpha f(x) \in \R^{N(\alpha)}$ \em (approximate differential) \em such
that
\begin{equation}\label{eqApDif1}
\ap\lim_{y\rightarrow x}\frac{|f(y)-f(x)-
L^\alpha f(x)({\bf x}_\alpha(y)- {\bf
x}_\alpha(x))|}{d(x,y)}=0,
\end{equation}
i.e. for every $\veps>0$ the set
\begin{equation}\label{good}
A_{x,\veps}=\Big\{y:\frac{|f(y)-f(x)-L^\alpha
f(x)\cdot({\bf x}_\alpha(y)-{\bf
x}_\alpha(x))|}{d(x,y)}<\veps\Big\}
\end{equation}
has $x$ as a density point
\end{defn}

The following lemma shows that the approximate differential is well-defined, in the sense that if there exists such vector
$L^\alpha f(x)$ satisfying \eqref{eqApDif1} then
it is unique for almost all points $x\in X_\alpha$. Thus, redefining (if necessary) the given measurable differentiable structure
on a set of measure zero we get the structure with respect to which the approximate differential is always unique.
\begin{lem}
Let $\{(X_{\alpha},\bf{ x}_{\alpha})\}_{\alpha\in \Lambda}$ be an approximate differentiable structure defined on a metric measure space $(X,d,\mu)$.
Then for every $\alpha\in {\bf \Lambda}$ one can choose a subset $\tilde X_\alpha \subset X_\alpha$ such that $\mu(X_\alpha\setminus \tilde X_\alpha)=0$ and
for any function $f:X\to \mathbb{R}$ and every $x\in \tilde X_\alpha$ the following statement is true:
if there exist vectors $L^\alpha_1 f(x),\,L^\alpha_2 f(x)\in \mathbb{R}^{N(\alpha)}$
such that
\begin{align}
\label{ApDifUniq}\text{\em ap}\lim_{y\rightarrow x}\frac{|f(y)-f(x)- L^\alpha_i f (x)\cdot({\bf x}_\alpha(y)- {\bf x}_\alpha(x))|}{d(x,y)}  = 0,\quad i=1,2,
\end{align}
then $L^\alpha_1 f(x) = L^\alpha_2 f(x)$.
\end{lem}

\begin{proof} 
The definition of the approximate differentiable structure implies that function $g\equiv 0$ has
a unique approximate differential on a set $\tilde X_\alpha$, such that $\mu(X_\alpha \setminus \tilde X_\alpha)=0$.

Assume that there is a function $f : X \to \R$ and a point $x
\in \tilde X_\alpha$ such that two different vectors $L^\alpha_1 f(x)$ and
$L^\alpha_2 f(x)$ satisfy \eqref{ApDifUniq}. By the definition of
the approximate limit, there exist sets $A_1,A_2\subset X_{\alpha}$ for
which $x$ is a density point and
\begin{equation*}
\lim_{\substack{y\to x\\ y\in A_i}}
\frac{|f(y)-f(x)-L^{\alpha}_i f(x)\cdot({\bf x}_{\alpha}(y)-{\bf
x}_{\alpha}(x))|}{d(y,x)}= 0,
\end{equation*}
for $i=1,2$. By the triangle inequality, we have that
\begin{equation}\label{eq}
\lim_{\substack{y\to x\\ y\in A_1\cap A_2}}
\frac{|(L^{\alpha}_1f-L^{\alpha}_2 f)(x)\cdot({\bf x}_{\alpha}(y)-{\bf
x}_{\alpha}(x))|}{d(y,x)}= 0.
\end{equation}
%

Since $\tilde X_\alpha$ is a set where $g\equiv 0$ has
a unique approximate differential, we have
\[
(L^{\alpha}_1f-L^{\alpha}_2 f)(x)=0,
\]
as required.
\end{proof}


In what follows we will prove that the approximate differential is a measurable function. We will need the following technical lemma.

\begin{lem}
\label{aplimsup}
Let $(X,d,\mu)$ be a metric measure space
and let $g: X \times X \to \R$ be a $\mu \otimes \mu$-measurable function.
Then $ x\mapsto {\rm ap}\limsup_{y\to x} g(x,y)$ is $\mu$-measurable.
\end{lem}

Recall that ${\rm ap}\limsup_{y\to x}F(y)$ is the infimum of the set of numbers $a\in\mathbb{R}$ for which the set $\{y\in X:\,F(y)>a\}$ has density zero at the point $x$. 

\begin{proof}

The proof is an immediate adaptation of the proof of the analogous statement in the Euclidean case \cite[3.1.3(2)]{F} and is based on 
the fact that for any $\mu \otimes \mu$-measurable set $S$ and any fixed $\varepsilon,\delta >0$, the set
\begin{align}
\label{meas}
\bigcap_{0<r<\delta} \big \{ x \in X \ | \ \mu(\{y \ |  \ (x,y) \in S, \ y \in B(x,r)\}) < \varepsilon \mu(B(x,r)) \big \}
\end{align}
is $\mu$-measurable. To prove the measurability of the set defined above, one need to use the fact, that for every $r>0$ the function $f(x) = \mu(B(x,r))$ is lower semicontinuous and, hence, measurable. 

To obtain the measurability of the function $x\mapsto {\rm ap}\limsup_{y\to x} g(x,y)$ it is enough to use the above observation for the set $ S:= \{(x,y) \ | \ g(x,y) > t \} $ for any $t\in\R$.
\end{proof} 

Now we are ready to prove measurability of the approximate differential.

\begin{lem}
\label{measurability}
Let $(X,d,\mu)$ be a metric measure space that supports an approximate differentiable structure.
If $f:X\to\R$ is a measurable function which is approximately
differentiable at $\mu$-almost every $x\in X_\alpha$, then the
approximate differential $L^\alpha f:X_\alpha\to \R^{N(\alpha)}$ is
$\mu$-measurable on $X_\alpha$.
\end{lem}

Here the value $L^\alpha f(x)$ is given
by Definition \ref{DefApDif1}, if $x$ is a point of approximate
differentiability of $f$, and $L^\alpha f(x)=\bf{0}$ otherwise.

\begin{proof}
To prove that function $l=L^\alpha f$ is measurable, we show that
$l^{-1}(K)$ is a measurable set for each compact $K\subset \R^{N(\alpha)}$.
Let $K$ be a compact set. Denote by
\[
A_x(\lambda)= {\rm ap}\limsup_{y\to
x}\frac{|f(y)-f(x)-\lambda\cdot({\bf x}_{\alpha}(y)-{\bf
x}_{\alpha}(x))|}{d(x,y)},\quad \lambda\in\R^{N(\alpha)}.
\]

Observe that, for every $x$, the function $\lambda\mapsto A_x(\lambda)$ is continuous. Indeed, since
\[
{\rm ap}\limsup_{y\to x}|g(y)+h(y)|\le {\rm ap}\limsup_{y\to
x}|g(y)|+{\rm ap}\limsup_{y\to x}|h(y)|,
\]
we have
\[
|A_x(\lambda)-A_x(\lambda')|\le{\rm ap}\limsup_{y\to
x}\frac{|(\lambda-\lambda')\cdot({\bf x}_{\alpha}(y)-{\bf
x}_{\alpha}(x))|}{d(x,y)}\le C|\lambda-\lambda'|
\]
for any $\lambda,\,\lambda'\in \R^{N(\alpha)}$.
Set
\[
E=\{x\in X_\alpha:\quad \exists \lambda\in K\quad \text{such
that}\quad A_x(\lambda)=0\}.
\]
Note that $E$ coincides with $l^{-1}(K)$. To check that $E$ is measurable, fix a dense countable subset
$K'$ of $K$. Then by the continuity of $A_x$ and the density of
$K'$ in $K$, we have
\[
E=\{x\in X_\alpha:\quad \exists (\lambda_n)_{n\in\N}\subset
K',\,\lambda\in K \quad \text{such that}\quad A_x(\lambda_n)\to
0\,\text{as}\,\lambda_n\to\lambda \}.
\]
Consequently, we can write $E$ as
\[
E=\bigcap_{n\in\N}\bigcup_{\lambda\in K'}\{x\in X_\alpha:\quad
A_x(\lambda)<\frac{1}{n} \}.
\]

To finish it remains to check that the function $x\mapsto A_x(\lambda)$ is measurable for each $\lambda\in\R^{N(\alpha)}$. This follows from Lemma \ref{aplimsup}.
\end{proof}

Next observe that the notion of approximate differentiability does not depend on the choice of the approximate differentiable structure.
\begin{lem}
\label{Lindstr}
Let $(X,d,\mu)$ be a metric measure space that admits an approximate differentiable structure $\{(X_\alpha,\bf{x}_\alpha)\}_{\alpha\in \Lambda}$.
If $f: X \to \R$ is approximately differentiable $\mu$-a.e. on $X$ with respect to $\{(X_\alpha,\bf{x}_\alpha)\}_{\alpha\in \Lambda}$ then
it is approximately differentiable almost everywhere with respect to any
approximate differentiable structure defined on $X$.
\end{lem}

\begin{proof} Let $\{(X_\alpha, \bf{x}_\alpha)\}_{\alpha\in \Lambda}$ and $\{(Y_\beta, \bf{y}_\beta)\}_{\beta\in B}$ be
two approximate differentiable structures defined on $(X,d,\mu)$.
We will write
$L^{\alpha}_{\bf x}f$ for the approximate differential of $f$ with respect to (w.r.t)
$\{(X_\alpha, \bf{x}_\alpha)\}_{\alpha\in \Lambda}$ at $ x\in X_\alpha $.


First we notice that for fixed $x$ the real valued function $g_x(\cdot) =L^{\alpha}_{\bf x} f(x) \cdot {\bf x}_\alpha(\cdot)$ is Lipschitz continuous
on $X$, thus it is approximately differentiable $\mu$-a.e. w.r.t. $\{(Y_\beta, \bf{y}_\beta)\}_{\beta\in B}$.
Moreover the set of points where $g_x$ is approximately differentiable does not depend on the choice of $x$.

Let $\beta\in \bf{B}$. For almost every $x\in X_\beta$ we can choose
$\alpha\in \bf{\Lambda}$ such that $x\in X_\alpha$ and $x$ is a point of approximate differentiability of $f$ w.r.t. $(X_\alpha, \bf{x}_\alpha)$.
Thus $\mu$-a.e. $x\in X_\beta$ we have
\begin{align*}
|f(y) - f(x) - L_{\bf y}^\beta g_x(x)({\bf y}_\beta(y) -  {\bf y}_\beta(x))| &\le  |f(y) - f(x) - L^\alpha_{\bf x} f(x)({\bf x}_\alpha(y) - {\bf x}_\alpha(x))| \\
& 
+ |g_x(y) - g_x(x) - L^\beta_{\bf y}g_x(x) ({\bf y}_\beta(y) - {\bf y}_\beta(x))|.
\end{align*}

Obviously the set
\[
A_{x,\varepsilon}:=\{y\in B(x,r): \ |f(y) - f(x) - L^\beta_{\bf y}g_x(x)({\bf y}_\beta(y) - {\bf y}_\beta(x))|<\varepsilon{d(x,y)} \}
\]
contains the intersection of the sets
\[
\{y\in B(x,r): \ |f(y) - f(x) - L^\alpha_{\bf x} f(x)({\bf x}_\alpha(y) - {\bf x}_\alpha(x))| <\frac{\varepsilon}{2}{d(x,y)} \}
\]
and
\[
\{y\in B(x,r):  \ |g_x(y) - g_x(x) - L^\beta_{\bf y}g_x(x) ({\bf y}_\beta(y) - {\bf y}_\beta(x))|<\frac{\varepsilon}{2}{d(x,y)} \}.
\]
Therefore each $x$ which is a point of approximate differentiability of $f$ w.r.t. $(X_\alpha, {\bf x}_\alpha)$ and a point of approximate
differentiability of
$L^\alpha_{\bf x} f(x)\cdot\bf{x}_\alpha$ w.r.t. $(Y_\beta, \bf{y}_\beta)$ is a point of density of $A_{x,\varepsilon}$. We conclude that $f$ is approximately differentiable a.e. on $X$
with respect to $\{(Y_\beta, \bf{y}_\beta)\}_{\beta\in B}$.
\end{proof}

\begin{rem}\label{weakstrong}
Definition \ref{DefApDif1} of approximate differentiability  makes sense whenever the underlying space supports an approximate differentiable structure. If we
additionally assume that the measure $\mu$ is doubling then an
approximate differentiable structure turns out to be a strong
measurable differentiable structure as well, see \cite[Prop.
3.5]{Ke2}. Since the results presented later on are formulated under the
assumption that the measure $\mu$ is doubling we, in fact, deal
with a strong measurable differentiable structure.

\end{rem}

\section{Characterization of approximate differentiability} \label{Characterization}
\subsection{Whitney-type characterization of approximate differentiability}
The proof of the following theorem is strongly inspired by the original proof of Whitney for the case $X=\R^n$, see \cite[Theorem 1]{Wh}.


\begin{thm}\label{main}
Let $(X,d,\mu)$ be a complete metric measure space, where $\mu$ is a doubling measure and let $\{(X_\alpha,\bf{ x}_{\alpha})\}_{\alpha\in \Lambda}$ be an approximate differentiable structure on $(X,d,\mu)$.
Suppose that $E\subset X$ and $f:E\to\R$ is a $\mu$-measurable function. Then the following conditions are equivalent:
\begin{itemize}
\item[$\text{\em(a)}$] $f$ is an approximately differentiable $\mu$-a.e. in $E$.
\item[$\text{\em(b)}$]
For any $\veps>0$ there is a closed set $F\subset E$ such that $\mu(E\setminus F)<\veps$ and $f_{|F}$ is locally Lipschitz.
\item[$\text{\em(c)}$] $f$ induces a Luzin decomposition of $E$, that is,
\begin{equation}\label{LuzinDecomposition}
E=\bigcup_{j=1}^{\infty}E_j\cup Z,
\end{equation}
where $E_i$ are measurable sets, $f_{|E_i}$ are Lipschitz
functions and $Z$ has measure zero.
\end{itemize}
\end{thm}

\begin{rem}
When $E\subset X$ is a bounded set, condition $(b)$ can be replaced by:

$(b')$ {\it For any $\veps>0$ there exists a closed set $F\subset E$ and a Lipschitz function $g:X\to \R$ such that $\mu(E\setminus F)<\veps$ and $f_{|F}=g$.}

\noindent
To show that the function $f$ is globally Lipschitz on $F$, one needs to notice that, since $X$ is proper, the set $F$ is compact.
Then one can extend $f$ to the whole space by standard arguments.

\end{rem}


\begin{proof}[Proof of Theorem \ref{main}] Without loss of generality, we can consider all coordinate functions ${\bf x}_\alpha$ to be Lipschitz continuous with a Lipschitz constant equal to one, since clearly $\{(X_\alpha, \frac{\bf{ x}_{\alpha}}{\textrm{LIP}({\bf x}_{\alpha})})\}_{\alpha\in {\bf\Lambda}}$ is an approximate differentiable structure on $X$, and $f$ is approximately differentiable w.r.t. the structure. Here $\textrm{LIP}({\bf x}_{\alpha})$ denotes the Lipschitz constant of ${\bf x}_{\alpha}$.


Let $f$ be approximately differentiable $\mu$-a.e. in $E$.
We can assume that the sets $X_\alpha$ are pairwise disjoint and extend $L^\alpha f$ by zero outside $X_\alpha$.  Denote by $N$ the bound on the dimension given by Theorem \ref{Cheeger}.
Consider $L^\alpha f(x)$ as vectors in $\mathbb{R}^N$ (we extend the vector with zeros when necessary) and let $Lf=\sum_{\alpha}L^\alpha f$.
If a function is Cheeger differentiable $\mu$-a.e. on $X$ the analogue construction would give a ``gradient'' for $f$.
This construction is quite standard in the literature, see e.g. \cite{Bj}, \cite{BBS2}.


{(a)$\Longrightarrow$(b)}

First, assume that $E\subset X$ is a bounded set.

Define
\[D=\{x\in E: f\text{ is approximately differentiable at } x\}.
\]

First we show that for any $\veps>0$ there exists a closed set $F=F_{\veps}\subset D$, $\delta>0$ and $L>0$ such that
$\mu(D\setminus F)<\veps$ and
$$
|f(x)-f(y)|\leq L\, d(x,y) \text{ for each } x,y\in F,\, d(x,y)<\delta.
$$

Since $\mu$ is a doubling measure,
we have for any $r>0$, $x,y\in X$ such that $d(x,y)\leq r/2$ that
\begin{equation}\label{eq0}
\mu(B(x,r)\cap B(y,r))\geq \mu(B(x,r/2))\geq 2\,a\,\mu (B(x,r)),
\end{equation}
where $a=1/2C_{\mu}$, and $C_{\mu}$ denotes the doubling constant.
\noindent

For each $\eta>0$ define the following sequence of functions:
\[
\psi^\eta_i(x)=\mu(B(x,1/i)\setminus A_{x,\eta})\quad x\in
D,\quad i\in\N,
\]
where $A_{x, \eta}$ is given by formula \eqref{good}. It is clear
that for each $i\in\N$, the function $\psi^\eta_i(x)$ is measurable in $x$ for fixed $\eta$.
Moreover, for each $\eta>0$ and $x\in D$ one has
\begin{equation}\label{eq1}
\phi_i^\eta(x)=\frac{\psi^\eta_i(x)}{\mu(B(x,1/i))}\to
0\quad\text{ as }\quad i\to\infty.
\end{equation}

Next set $\eta=1$. By Luzin's and Egoroff's theorem there exists a closed set $F\subset E$ such that
\begin{itemize}
\item[$(i)$]$\mu(E\setminus F)=\mu(D\setminus F)<\veps$,
\item[$(ii)$] $Lf_{|F}$ is continuous, moreover, since $X$ is proper, $Lf_{|F}$  is bounded in $F$, i.e. $|{Lf}_{|F}|\leq C$,
\item[$(iii)$] and $\phi^1_i\to 0$ uniformly on $F$.
\end{itemize}
Now choose $i_0$ such that
\begin{equation}\label{eq3}
\phi^1_i(x)<\frac{a}{C_\mu},\quad x\in F,\quad i\geq i_0,
\end{equation}
where $C_\mu\geq 1$ is the doubling constant.

Fix $x,y\in F$ such that $d(x,y)<1/(2i_{0})$. 
Points $x$ and $y$ may belong to two different charts, thus we write $x\in X_{\alpha} $ and $y\in X_{\beta}$.
Choose $i\geq i_0$ such that
\[
1/(2i+2)<d(x,y)\le 1/(2i).
\]
For such $i$ we have that $i \geq i_0$ and by \eqref{eq3}, we get that
\[
\psi^1_i(y)<\frac{a}{C_\mu}\mu(B(y,1/i))\le\frac{a}{C_\mu}\mu(B(x,2/i))\le a\mu(B(x,1/i)).
\]
Hence
\begin{equation}\label{eq4}
\psi^1_{i}(x)< a \mu(B(x,1/i))\quad\text{and}\quad\psi^1_{i}(y)< a \mu(B(x,1/i)).
\end{equation}

Combining \eqref{eq0} and \eqref{eq4}, we deduce that there exists a point $z\in B(x,1/i)\cap B(y,1/i)$ which does not belong either to
$B(x,1/i)\setminus A_{x,1}$ or $B(y,1/i)\setminus A_{y,1}$. For such point $z$, we have that $d(x,z)<1/i$, $d(y,z)<1/i$,
\begin{equation}\label{eq5}
\frac{|f(z)-f(x)-L^{\alpha}f(x)\cdot({\bf x}_{\alpha}(z)-{\bf
x}_{\alpha}(x))|}{d(z,x)}<{1}
\end{equation}
and
\begin{equation}\label{eq6}
\frac{|f(z)-f(y)-L^{\beta}f(y)\cdot({\bf x}_{\beta}(z)-{\bf
x}_{\beta}(y))|}{d(z,y)}<1.
\end{equation}
By combining  \eqref{eq5}, \eqref{eq6}, $d(y,z)<5d(x,y)$ and $d(x,z)<5d(x,y)$ we obtain the inequality
\begin{equation*}
\begin{split}
|f(y)-f(x)|\leq &|f(z)-f(x)-L^{\alpha}f(x)\cdot({\bf x}_{\alpha}(z)-{\bf
x}_{\alpha}(x))|\\
&+|f(z)-f(y)-L^{\beta}f(y)\cdot({\bf x}_{\beta}(z)-{\bf
x}_{\beta}(y))| \\
&+|L^{\alpha}f(x)\cdot({\bf x}_{\alpha}(z)-{\bf
x}_{\alpha}(x))|+|L^{\beta}f(y)\cdot({\bf x}_{\beta}(z)-{\bf
x}_{\beta}(y))|\\
&\leq  10d(x,y)+ C d(x,y),
\end{split}
\end{equation*}
which shows that $f_{|F}$ is a locally Lipschitz function and finishes the proof of the implication for the case in which $E \subset X$ is a bounded set.

Let now $E$ be an arbitrary subset of $X$. Fix any point $x_0\in X$ and consider a family of open balls
$
B_j=B(x_0,j),\, j=1,2,\ldots,
$
covering $X$.

Apply now the above reasoning to get closed sets $F_j\subset E\cap B_j$ such that
\[
\mu\big((E\cap B_j)\setminus F_j\big)\le 2^{-j}\varepsilon
\]
and $f_{|F_j}$ are locally Lipschitz functions. Set
\[
F=X\setminus \bigcup\limits_{j=1}^{\infty}(B_j\setminus F_j),
\]
then
\[
\mu(E\setminus F)=\mu(E\cap F^c)=\mu\bigg(E\cap \bigcup\limits_{j=1}^{\infty}(B_j\setminus F_j)\bigg)\le\sum\limits_{j=1}^{\infty}\mu(E\cap (B_j\setminus F_j))\le\varepsilon.
\]
It is easy to see that $F$ is a closed set and $F\cap B_j\subset F_j$ for every $j=1,2,\dots$.
Hence, $F\subset E$ and $f_{|F}$ is locally Lipschitz.
The last observation follows from the fact that for every point of $F$ we can find a neighborhood of $x$ contained in some $B_j$ and $f_{|F_j}$ is
locally Lipschitz.

\noindent
(b)$\Longrightarrow$(c) For each $i \in \N$ there exists a closed set $F_i$ such that $\mu(E\setminus F_i)\le 1/i$ and $f_{|F_i}$ is locally Lipschitz.
Setting $\tilde F_k:= \bigcup_{i=1}^k F_i$ we obtain an ascending family of closed sets, such that $f$ is locally Lipschitz on each of its members. We define $Z:=\bigcap_{k=1}^\infty(E\setminus \tilde F_k)$. Then,
\[
\mu(Z) = \lim_{k\to \infty} \mu(E\setminus \tilde F_k) = 0.
\]
Let $\{B_i\}_{i=1}^\infty$ denote a countable family of balls covering $X$. Then
\[
E= Z \cup \bigcup_{i=1}^\infty \bigcup_{k=1}^\infty \tilde F_k \cap B_i=Z \cup\bigcup_{j=1}^\infty F'_j,
\]
where the family $\{F'_j\}^\infty_{j=1}$ is obtained just by renumerating the family $\{F_k \cap B_i\}$. Observe that $f_{|F'_j}$ is Lipschitz.
To finish the proof we take the disjoint sets as follows: $E_{1}:=F_1'$ and $E_{j}:= F_j'\setminus\bigcup_{m<j}E_m$, $j>1$.

\noindent
(c)$\Longrightarrow$(a)
If decomposition \eqref{LuzinDecomposition} holds, then the restriction $f_{|E_{i}}$ is Lipschitz for every $i$. Using Mc Shane's theorem we can extend $f_{|E_{i}}$ to a Lipschitz function $\tilde f_i$ defined on the whole space $X$.
By the definition of the approximate differentiable structure, $\tilde f_i$ is $\mu$-a.e approximately differentiable on $X$.
Since every point of a measurable set $E_i$ is its point of density, $f$ is also $\mu$-a.e approximately differentiable on $E_{i}$.
\end{proof}

\subsection{Stepanov-type characterization} \label{stepasection}

The following Stepanov-type theorem shows that an approximate local growth condition on a function guarantees its approximation by Lipschitz functions in Luzin sense.

\begin{thm}\label{StepanovEqivalLuzin}
Let $\mu$ be a doubling Borel measure. A $\mu$-measurable function $f:E\to\mathbb{R}$ defined on a measurable subset $E\subset X$  satisfies the condition
\begin{equation}\label{ApStepanovCondition}
{\rm ap}\limsup_{y\to x}\frac{|f(y)-f(x)|}{d(x,y)}<\infty
\end{equation}
for $\mu$-a.e $x$ in $E$ if and only if for any $\varepsilon>0$ there is a closed set $G\subset E$ such
that $\mu(E\setminus G)\le\varepsilon$ and $f$ is locally Lipschitz on $G$.
\end{thm}
The proof is an adaptation of arguments in \cite[Theorem 3.1.8]{F} to the metric setting.

\begin{proof}
First assume that condition \eqref{ApStepanovCondition} holds. Define for each positive integer $j$ the set
\[
Q(u,r,j)=B(u,r)\cap \{x:\,x\notin E\,\,\text{or}\,\, |f(x)-f(u)|> jd(x,u)\},
\]
whenever $u\in E$ and $r>0$.
Define also the set
\[
A_j=E\cap \{u:\, \mu(Q(u,r,j))<a\mu(B(u,r))\,\,\,\text{for}\,\,0<r<1/j\},
\]
where $a>0$ is some constant for which \eqref{eq0} holds.
Then each set $A_j$ is measurable, which follows from measurability of the sets defined by \eqref{meas}),  and
\begin{equation}\label{eqInprofOfStepanovTh}
\mu\bigg(E\setminus\bigcup\limits_{j=1}^\infty A_j\bigg)=0.
\end{equation}

Observe that if $u,v\in A_j$ and $d(u,v)<1/2j$, then
\[|f(u)-f(v)|\le 4jd(u,v).\]
Indeed, set $r=2d(u,v)$, then by the definition of the sets $Q$ and by inequality \eqref{eq0}
\[
\mu(Q(u,r,j)\cup Q(v,r,j))<a(\mu(B(u,r))+\mu(B(v,r)))\le \mu(B(u,r)\cap B(v,r)).
\]
Thus, we can choose point $x\in\big(B(u,r)\cap B(v,r)\big)\setminus \big(Q(u,r,j)\cup Q(v,r,j)\big)$, and we have
\[
|f(u)-f(v)|\le|f(u)-f(x)|+|f(v)-f(x)|\le j(d(x,u)+d(x,v))\le 2jr=4jd(u,v).
\]

It follows from the last inequality that $f$ is locally Lipschitz on every $A_j$.
Since the sequence of sets $A_j$, $j=1,2,\dots$ is increasing, the measure $\mu$ is Borel regular and equality \eqref{eqInprofOfStepanovTh} holds,
for any $\varepsilon>0$ we can choose a closed set $G\subset E$ such that $\mu(E\setminus G)\le \varepsilon$ and $f|_G$ is a locally Lipschitz function.


Let us show the reverse implication. Let $G$ be a closed set such that $\mu(E\setminus G)\le \varepsilon$ and $f|_G$ is a locally Lipschitz function.
Then $X\setminus G$ has density zero at $\mu$-almost every point of $G$. Thus, \eqref{ApStepanovCondition} holds $\mu$-a.e. in $G$ and
the fact $\varepsilon$ can be chosen arbitrary small finishes the proof.
\end{proof}

%

As a corollary of Theorem \ref{main} and Theorem \ref{StepanovEqivalLuzin}, we get the following characterization of approximate differentiability.

\begin{cor}\label{StepanovTypeCharact}
Under the hypothesis of Theorem \ref{main}, a function $f:X\to \mathbb{R}$ is approximately differentiable $\mu$-a.e. in a bounded measurable subset $E\subset X$ if and only if
\[
{\rm ap}\limsup_{y\to x}\frac{|f(y)-f(x)|}{d(x,y)}<\infty,\,\,\,\mu\text{-a.e. in }E.
\]
\end{cor}

A similar integral local growth condition is used in \cite{R} to guarantee $L^p$-differentiability of a function. It is also mentioned in \cite[Remark 3.4]{R} that the technique used in \cite[Theorem 3.3]{R} can be adapted for the notion of approximate differentiability.

As mentioned before, approximate differentiability is a much
weaker property than differentiability. However, if it is the case
that both, the approximate differential and Cheeger differential
exist almost everywhere, they should coincide. Therefore it is
interesting to search for additional conditions of global and
infinitesimal character which imply Cheeger differentiability
almost everywhere.

The following Stepanov differentiability theorem in metric measure
spaces was proved by Balogh-Rogovin-Z\"urcher in \cite{BRZ}.

\begin{thm}{\em{\bf \cite{BRZ}}}\label{ThBRZ}
Let $(X,d,\mu)$ be a metric space endowed with a doubling Borel measure $\mu$.
Assume that there is a strong measurable differentiable structure for
$(X,d,\mu)$.
Then a function $f:X\to\mathbb{R}$ is $\mu$-a.e. Cheeger
differentiable in the set
\begin{equation}\label{StepanovCondition}
\Big\{x:\,\limsup_{y\to x}\frac{|f(y)-f(x)|}{d(x,y)}<\infty\Big\}.
\end{equation}
\end{thm}

The proof of Theorem \ref{ThBRZ} is based on Maly's
proof of Stepanov's theorem in the Euclidean case, see \cite{Ma}.
Note that Stepanov differentiability theorem in $\R^n$ can be also
derived from its approximate analogue, see e.g. \cite{F}.
The same arguments work in metric spaces. Thus, one can obtain  an alternative proof for Theorem \ref{ThBRZ} combining Corollary \ref{StepanovTypeCharact} and the version of \cite[Lemma 3.1.5]{F} adapted to the metric measure setting.

\section{Differentiability properties of Sobolev functions}\label{differentiability}

In this section, we show that the approximate differentiability of Sobolev functions and $BV$ functions follows easily from the Stepanov-type characterization of approximate differentiability. The results in this section are basically known, but our approach gives another point of view.

First let us notice that if we have a Lipschitz-type pointwise estimate for a function, then we have an approximate local growth condition on $f$ as in Theorem \ref{StepanovEqivalLuzin}. Namely, if $f: X\to \R$ is a $\mu$-measurable function for which there exists a $\mu$-measurable function $g: X\to\R$ with
\begin{equation}\label{haj}
|f(x)-f(y)|\leq d(x,y)(g(x)+g(y))\qquad\text{$\mu$-a.e.},
\end{equation}
then
\begin{equation}\label{apsup}
{\rm ap}\limsup_{y\to x}\frac{|f(y)-f(x)|}{d(x,y)}<\infty,\,\,\,\mu\text{-a.e. in }X.
\end{equation}
Indeed, notice first that by Luzin theorem, $g$ is approximately continuous $\mu$-a.e. Divide both sides of the inequality by $d(x,y)$ and take approximate supremum limits when $y\to x$ to get \eqref{apsup}.

There are several generalizations of
classical Sobolev spaces to the setting of arbitrary metric measure
spaces.

Haj{\l}asz-Sobolev spaces $M^{1,p}(X)$ were defined in \cite{H}  as the functions in $L^p(X)$ for which there exists a positive function $g\in L^p(X)$  satisfying inequality \eqref{haj}. It follows from the discussion above that under the hypothesis of Theorem \ref{main}, Haj{\l}asz-Sobolev functions are approximately differentiable almost everywhere.

Using the notion of upper gradient (and
more generally weak upper gradient), Shanmugalingam
in \cite{Sh} introduced Newtonian spaces
$N^{1,p}(X)$ for $1\leq p\leq\infty$. A non-negative Borel function $g$ on $X$ is a \em $p$-weak upper
gradient \em of an extended real-valued function $f$ on $X$ if $
|f(\gamma(a))-f(\gamma(b))|\leq\int_{\gamma}g$
for all \ rectifiable curves $\gamma:[a,b]\to X$ except for a family of zero $p$-modulus. See \cite{Sh} for the definition of modulus of a family of curves.

Let $\widetilde{N}^{1,p}(X,d,\mu)$, where $1\leq p\leq\infty$, be the
class of all $p$-integrable functions on $X$ for which there
exists a $p$-weak upper gradient in $L^p(X)$. For
$f\in\widetilde{N}^{1,p}(X,d,\mu)$, we define
$$
\|f\|_{\widetilde{N}^{1,p}}:=\|f\|_{L^p}+\inf_{g}\|g\|_{L^p},
$$
where the infimum is taken over all $p$-weak upper gradients $g$ of
$f$. Now, we define in $\widetilde{N}^{1,p}(X,d,\mu)$  an equivalence
relation  by $f_1\sim f_2$ if and only if
$\|f_1-f_2\|_{\widetilde{N}^{1,p}}=0$. Then the space $N^{1,p}(X,d,\mu)=N^{1,p}(X)$
is defined as the quotient $\widetilde{N}^{1,p}(X,d,\mu)/\sim$.

As shown by Haj{\l}asz and Koskela \cite{HKo}, if we have a pair of functions $(f,g)$ that satisfies a $p$-Poincar\'e inequality  \eqref{poinc}, then we have the following pointwise estimate
\begin{equation}\label{maxhaj}
|f(x)-f(y)|\leq C d(x,y)[(M_{2\sigma d(x,y)}g^p(x))^{1/p}+(M_{2\sigma d(x,y)}g^p(y))^{1/p}],
\end{equation}
for $\mu$-a.e $x,y\in X$ and for some constants $C,\sigma>0$. Here $M_{R} f$ is defined by
\[
M_{R} f(x):=\sup_{0<r\leq R}\jint_{B(x,r)}|f(y)|d\mu(y),
\]
and notice that $M_{R} f(x)\leq Mf(x)$, where $Mf$ is the standard Hardy-Littlewood maximal function. Actually, if the space supports a doubling measure and a $p$-Poincar\'e inequality, with $p>1$, Newtonian spaces are characterized by  \eqref{maxhaj}. Moreover, under these hypotheses, Newtonian spaces coincide with Hajlasz-Sobolev spaces. If $X$ is only known to support an $\infty$-Poincar\'e inequality, the space of Lipschitz functions coincides with $N^{1,\infty}(X)$ (see \cite[Theorem 4.6]{DJS}).

Stepanov-type characterization can be also used to prove that $BV$ functions on metric spaces are approximately differentiable almost everywhere. See the work by Miranda \cite{M} for the corresponding definition of $BV$ functions. Very recently, Lahti and Tuominen \cite{LaTu} have shown that a similar pointwise estimate as in \eqref{maxhaj} holds for $BV$ functions assuming that the space supports a $1$-Poincar\'e inequality. Namely, if $f\in BV(X)$, there exists a constant $\sigma\geq 1$ such that
\begin{equation}
|f(x)-f(y)|\leq C d(x,y)[M_{2\sigma d(x,y),\|Df\|}(x)+M_{2\sigma d(x,y),\|Df\|}(y)],
\end{equation}
for $\mu$-a.e. $x,y\in X$, where $C$ is a constant depending only on the doubling constant and the constants involved in the Poincar\'e inequality. Here $M_{2\sigma d(x,y),\|Df\|}$ denotes the restricted maximal function of the measure $\|Df\|$, that is,
$$
M_{R,\|Df\|}(x):=\sup_{0<r\leq R}\frac{\|Df\|(B(x,r))}{\mu(B(x,r))},
$$
where $\|Df\|$ denotes the total variation of the measure $\mu$.

\begin{cor}\label{sobolev}
If $(X,d,\mu)$ is doubling and supports a $p$-Poincar\'e inequality, then Newtonian functions, Haj{\l}asz-Sobolev functions and $BV$ functions  are approximately differentiable $\mu$-a.e.
\end{cor}

Notice that the assumption of a Poincar\'e inequality cannot be dropped from the hypothesis in the Newtonian case or in the $BV$ case. For example, if the space has no rectifiable curves, except for the constant ones, then $N^{1,p}(X)=L^p(X)$ or $BV(X)=L^1(X)$ and therefore it could happen that a function in $N^{1,p}(X)$ is nowhere differentiable, nor approximately differentiable. On the other hand, when one uses Haj{\l}asz approach, it is enough to assume that the space admits a strong measurable differentiable structure to reach the conclusion.

These result can be also deduced from existing results in literature. Bj\"{o}rn \cite{Bj} has shown that if $(X,d,\mu)$ is doubling and supports a $p$-Poincar\'e inequality, then for each function $f\in N^{1,p}(X)$ and $\mu$-a.e. $x\in X$,
$$
\limsup_{r\to 0}\frac{1}{r}\jint_{B(x,r)}|f(y)-f(x)-df^{\alpha}(x)\cdot({\bf x}_{\alpha}(y)-{\bf
x}_{\alpha}(x))|d\mu(y)=0,
$$
in other words, $f$ is $L^1$-differentiable.
For uniformly perfect spaces equipped with a doubling measure, $L^1$-differentiability implies approximate differentiability. For a proof of this fact see \cite[Prop 3.4]{Ke2}. Notice that a space supporting a Poincar\'e inequality is connected and thus also uniformly perfect.


In \cite{R} it is proved that $BV$ functions are $L^1$-differentiable $\mu$-a.e. As a direct consequence we deduce as well that $BV$ functions are approximately differentiable $\mu$-a.e.

Notice that the approximate differentiability is a weaker notion than $L^p$-differentiability. In particular, the definition of the approximate differentiability does not involve any integrability assumptions.

\section{Approximate differentiability of the maximal function}\label{maxim}

Haj{\l}asz and Mal\'y  proved in \cite{HM} that, in the case of $X=\R^{n}$, approximate differentiability is preserved under the action of
\em
Hardy-Littlewood maximal operator \em
\[
Mf(x):=\sup_{r>0}\dashint_{B(x,r)}|f(y)|\,dy, \quad x\in X.
\]
It was recently shown by H. Luiro \cite{L} that, in the Euclidean case,
differentiability almost everywhere is also preserved under the
action of the maximal function.

On the other hand, in the setting of metric spaces endowed with a doubling measure, the maximal operator
does not preserve the regularity of a function in the same manner as in the Euclidean case.
Kinnunen proved in \cite{Ki}
that the Hardy-Littlewood maximal operator is bounded in $W^{1,p}(\R^n)$ for $1<p\leq\infty$. Notice that the case $p=\infty$ corresponds to the space of Lipschitz functions. On the other hand,
Buckley \cite{Bu} has shown that for a metric space with a doubling measure,
the maximal operator may not preserve Lipschitz and H\"older spaces.
In order to have a maximal function which preserves, for example, the Sobolev spaces on metric spaces, Kinnunen and Latvala \cite{KL}
constructed a maximal function based on discrete convolution (see also \cite{AK} and \cite{AK2}).

In the next theorem, we will show that the discrete maximal operator preserves also approximate differentiability. First, we define the discrete maximal operator.
Fix $r>0$. Let $B_{i}=B(x_{i},r)$, $i=1,2,\ldots$, be a collection of balls such that they cover $X$ and the balls $B(x_{i},r/2)$, $i=1,2,\ldots$, are pairwise disjoint. Let $\psi_{i}$ be a partition of unity subordinate to the covering $B_{i}$, $i=1,2,\ldots$, i.e. $0\leq \psi_{i}\leq 1$, $\supp \psi_{i}\subset B(x_{i},6r)$, $\psi_{i}\geq 1/C$ in $B(x_{i},3r)$, $\psi_{i}$ is Lipschitz with constant $L/r$ and $\sum_{i} \psi_{i} =1$. Then we define the discrete convolution of $f\in L^{1}_{\text{loc}}(X)$ by setting
\[
f_{r}(x)=\sum_{i=1}^{\infty}\psi_{i}(x)f_{B(x_{i},3r)}
\]
Let $r_{j}$ be an enumeration of the positive rationals. We define the discrete maximal function (which depends on the chosen covering)
\[
M^{*}f(x)=\sup_{j}|f|_{r_{j}}(x).
\]
Now we can state our theorem. Notice that the theorem holds for arbitrarily chosen covering definig $M^{*}f$.

\begin{thm}\label{maxdiff}
Let $(X,d,\mu)$ be a metric space equipped with a doubling measure $\mu$. Assume also that $X$ supports an approximate differentiable structure.
If $f\in L^1(X)$ is approximately differentiable $\mu$-a.e, then $M^{*}f$ is approximately differentiable $\mu$-a.e..
\end{thm}

The proof follows the ideas used in \cite{HM}. First, we consider the restricted maximal function $Mf_\veps$, $\veps>0$,
defined by the formula
\[
M^{*}_\veps f(x) :=  \sup_{r_{j}>\veps}|f|_{r_{j}}(x) .
\]

\begin{lem}
\label{trunlip}
Let $(X,d)$ be a metric measure space equipped with a doubling measure $\mu$. Assume also that $X$ supports an approximate differentiable structure.
If $f\in L^1(X)$ then $M^{*}_\veps f$, $\veps>0$, is approximately differentiable $\mu$-a.e. in $X$.
\end{lem}

\begin{proof}
We start by proving that for some constant $\tilde Q$, which depends only on the doubling constant, the following inequality holds
\begin{equation}
\label{lipmaxtrun1}
|M^{*}_\veps f(x) - M^{*}_\veps f(y)| \le \frac{\tilde Q}{\veps}d(x,y)(M^{*}_\veps f(x) + M^{*}_\veps f(y))\qquad \mbox{ for}\, \mu-{\text {a.e.}}\, x,y\in X.
\end{equation}
Notice first, that the claim clearly holds if $d(x,y)\geq\veps$, so we may assume that $d(x,y)<\veps$.
Fix $r>\varepsilon$. Let $I$ be the set of indexes $i$ such that $x$ or $y$ belong to $B(x_{i},6r)$. The doubling property implies that $|I|\leq C$ with a constant that only depends on the doubling constant. We also have
\[
|f|_{B(x_{i},3r)}\leq CM^{*}_{\veps}f(x)
\]
with a constant only depending on the doubling constant, see for example the proof of Lemma 3.1. in \cite{KL}. Thus we can conclude that
\[
\begin{split}
\left||f|_{r}(x)-|f|_{r}(y)\right|=&\big|\sum_{i\in I}(\psi_{i}(x)-\psi_{i}(y))|f|_{B(x_{i},3r)}\big|\\
\leq &\,C\, \frac{L}{r} d(x,y) M^{*}_{\veps}f(x)\\
\leq &\, C\,  \frac{L}{\veps} d(x,y) (M^{*}_{\veps}f(x) +M^{*}_{\veps}f(y)).
\end{split}
\]
By taking the supremum over all $r_{j}\geq \veps$, we obtain \eqref{lipmaxtrun1}.

Now it is enough to notice that the restricted maximal function is $\mu$-measurable,
hence by Luzin theorem it is approximately continuous and by \eqref{lipmaxtrun1}
\[
{\text{ap}}\limsup_{y \to x}\frac{|M^{*}_\veps f(x) - M^{*}_\veps f(y)|}{d(x,y)}
\le \frac{\tilde Q}{\veps}2 M^{*}_\veps f(x) < \infty \quad \mu\mbox{-a.e.\, in}\, X.
\]
Using Stepanov-type characterization (see Corollary \ref{StepanovTypeCharact}), we obtain that $M^{*}_\veps f$
is approximately differentiable $\mu$-a.e. in $X$.
\end{proof}

It would be interesting to know, whether Theorem \ref{maxdiff} holds for the standard Hardy-Littlewood maximal function as well. In the metric space setting this would require a totally different proof since estimates like \eqref{lipmaxtrun1} do not hold in spaces where the measure of balls does not behave nicely. Even with annular decay property, only H\"older type estimates are available.

Now we are ready to prove Theorem \ref{maxdiff}
\begin{proof}[Proof of Theorem \ref{maxdiff}]
First, we split the space into two parts
\[
X = \{x \ : \ M^{*}f(x) > |f|(x) \} \cup \{x \ : \ M^{*}f(x) = |f|(x) \}.
\]

Observe that 
if $f\in L^{1}(X)$
is approximately differentiable function at $\mu$-almost every point in $X$ then $|f|$ is approximately differentiable $\mu$-a.e. in $X$ as well.
This fact easily follows, for example, from Theorem \ref{main} on Whitney-type characterization of approximate differentiability.


Thus, the maximal function $M^*f$ is approximate differentiable $\mu$-a.e. on the second set.
Note also that since $\mu$-almost every point of $X$ is a Lebesgue point of $f$ (see e.g. \cite{He}, Theorem 1.8),
it is enough to show that $M^*f$ is approximately differentiable almost everywhere on the set
\[
A := \{x \ : M^{*} f(x) > |f(x)|  \ \textrm{and}  \  x  \ \textrm{is a Lebesgue point of} \ f \}.
\]
If $x \in A$ there exists a sequence $\{r_n\}_{n=1}^\infty$ such that
\[
\lim_{n\rightarrow \infty}|f|_{r_{n}}(x) =M^{*}f(x).
\]
The sequence $r_n$ is bounded (since $M^{*}f(x)>0$ and $f\in L^1(X)$) and we can find a convergent subsequence. Let us denote its limit by $r$.
Note that $r>0$, otherwise $M^{*}f(x)=|f|(x)$. Thus for each $x \in A$ there exists $k\in \mathbb{N}$ such that
$M^{*}f(x) =M^{*}_{1/k}f(x)$ and
\[
A\subset \bigcup_{k=1}^\infty \{x \ : \ M^{*}f(x) = M_{1/k}^{*}f(x) \}.
\]
By Lemma \ref{trunlip}, each maximal function $M^{*}_{1/k}f(x)$, $k\in \mathbb{N}$, is approximately differentiable $\mu$-a.e. in $X$.
and, since sets $\{x : \ M^{*}f(x) = M^{*}_{1/k}f(x) \}$ are measurable, $M^{*}f$
is approximately differentiable $\mu$-a.e. in $A$.
\end{proof}

\centerline{\sc Acknowledgements}
Part of this research was
conducted while the first and fourth author visited Aalto University; they wish
to thank the institution for its kind hospitality.
The authors would also like to thank Prof. Juha Kinnunen  for giving
comments and suggestions which have been helpful in improving the manuscript. The first author is partially supported by grant MTM2009-07848 (Spain).
The second, the third and the fourth authors are supported by the Academy of Finland (grants 252293, 250403 and 138738). The fourth author is partially supported by MNiSW Grant no N N201
397737, \emph{Nonlinear partial differential equations: geometric and variational problems}.


\begin{thebibliography}{999}
\addcontentsline{toc}{chapter}{References}
\renewcommand{\baselinestretch}{1}

\bibitem[AK]{AK} D. Aalto and J. Kinnunen: The discrete maximal operator in metric spaces. \em J. Anal. Math. \em {\bf111} (2010), 369--390.

\bibitem[AK2]{AK2} D. Aalto and J. Kinnunen: Maximal functions in Sobolev spaces. Sobolev spaces in mathematics. I, 25--67, Int. Math. Ser. (N. Y.), 8, Springer, New York, 2009.

\bibitem[AT]{AT}L. Ambrosio, P. Tilli: Topics on Analysis in
Metric Spaces. Oxford Lecture Series in Mathematics and its
Applications {\bf 25}. \em Oxford University Press, Oxford\em,
(2004).

\bibitem[BRZ]{BRZ}  Z. M. Balogh, K. Rogovin, T. Z\"urcher: The Stepanov Differentiability Theorem in Metric Measure
Spaces. \em J. Geom. Anal. \em {\bf14} no. 3, (2004), 405--422.

\bibitem[BV]{BV} S. G. Basalaev, S. K. Vodopyanov: Approximate Differentiability of Mappings of Carnot-Carath\'eodory Spaces. Preprint (2012).

\bibitem[BS]{BS} D. Bate, G. Speight: Differentiability, porosity and doubling in metric measure spaces. Preprint (2011).

 \bibitem[Bj]{Bj} J. Bj\"{o}rn: $L^q$-differentials for weighted Sobolev spaces. \em Michigan Math. J. \em {\bf47} (2000), no. 1, 151--161.

\bibitem[BB]{BB} A. Bj\"{o}rn, J. Bj\"{o}rn: Nonlinear Potential Theory on Metric Spaces. EMS Tracts in Mathematics, European Mathematical Society, Zurich (2011).


 \bibitem[BBS2]{BBS2} A. Bj\"{o}rn, J. Bj\"{o}rn, N. Shanmugalingam: The Dirichlet problem for p-harmonic functions on metric spaces. \em J. Reine Angew. Math.\em {\bf 556} (2003), 173--203.

 \bibitem[Bo]{Bo} B. Bojarski: Differentiation of measurable functions and Whitney-Luzin type structure theorems. Helsinki University of Technology Institute of Mathematics Research Reports (2009).


\bibitem[Bu]{Bu} S. M. Buckley: Is the maximal function of a Lipschitz function continuous? \em Ann. Acad. Sci. Fenn. Math. \em {\bf 24} (1999), 519--528.


\bibitem[Ch]{Ch} J. Cheeger: Differentiability of Lipschitz
Functions on metric measure spaces. \em Geom. Funct. Anal. \em,
\textbf{9} (1999), 428--517.

\bibitem[De]{De} A. Denjoy: Sur les fonctions d\'eeriv\'ees sommables. \em Bull. Soc. Math. France \em {\bf 43} (1916), 161--248.

\bibitem[DJS]{DJS}  E. Durand-Cartagena, J. A. Jaramillo, N. Shanmugalingam: The $\infty$-Poincar\'e inequality in metric measure spaces. \em Michigan Math. Journal. \em \textbf{60} (2011).

\bibitem[EG]{EG} L. C. Evans, R. F. Gariepy: Measure theory and fine properties of functions. Studies in Advanced Mathematics. CRC Press, Boca Raton, FL, 1992.

\bibitem[F]{F} H. Federer, Geometric measure theory. Springer-Verlag, New York, Heidelberg, 1969.


\bibitem[HM]{HM} P. Haj{\l}asz, J. Mal\'y: On approximate differentiability of the maximal function. \em Proc. Amer. Math. Society. \em{\bf138} (2010), 165--174.

\bibitem[H]{H} P. Haj{\l}asz: Sobolev spaces on metric-measure
spaces. \em Contemp. Math. \em {\bf338} (2003), 173--218.

\bibitem[HKo]{HKo} P. Haj{\l}asz, P. Koskela: Sobolev met Poincar\'e. \em Mem. Am. Math. Soc. \em{\bf145} (688)(2000).

\bibitem[HeKo]{HeKo}J. Heinonen, P. Koskela: Quasiconformal maps in metric spaces with controlled geometry. \em
Acta Math. \em {\bf181} (1998), 1--61.

\bibitem[He]{He} J. Heinonen, Lectures on Analysis on Metric Spaces. Universitext. Springer-Verlag, New York, 2001.

\bibitem[He2]{He2} J. Heinonen: Nonsmooth calculus. \em Bull. Amer. Math. Soc. \em {\bf44} (2007),
163--232.

\bibitem[Ke1]{Ke1}S. Keith: A differentiable structure for metric
measure spaces. \em Adv. Math. \em {\bf183} (2004) 271--315.

\bibitem[Ke2]{Ke2} S. Keith: Measurable differentiable structures and the Poincar\'e inequality. \em Indiana Univ. Math. J. \em {\bf53} (2004), no. 4, 1127--1150.

\bibitem[Ki]{Ki} J. Kinnunen: The Hardy-Littlewood operator of a Sobolev function. \em Israel J. Math. \em {\bf 100} (1997), 117-124.


\bibitem[KL]{KL}J. Kinnunen, V. Latvala: Lebesgue points for Sobolev functions on metric spaces. \em Revista Mat. Iberoamericana \em {\bf18} (2002), 105--132.

\bibitem[KlMa]{KlMa} B. Kleiner, J. Mackay: Differentiable structures on metric measure spaces: A Primer. Preprint. 2011.



\bibitem[LaTu]{LaTu} P. Lahti, H. Tuominen: A pointwise characterization of functions of bounded variation on metric spaces. Preprint (2012).

\bibitem[Li]{Li} F-C Liu: Luzin type property of Sobolev function. \em Indiana Univ. Mathematics Journal \em{\bf 26}(4) (1977), 645--651.


\bibitem[LT1]{LT1} F-C Liu, W-S Tai: Approximate Taylor Polynomials and differentiation of functions \em Topological Methods in Nonlinear Analysis \em {\bf 3} (1994), 189--196.

\bibitem[LT2]{LT2} F-C Liu, W-S Tai: Luzin properties and interpolation of Sobolev spaces. \em Topological Methods in Nonlinear Analysis \em {\bf 9} (1997), 163--177.

\bibitem[L]{L} H. Luiro: On the size of the set of non-differentiability points of maximal function. Preprint (2012).

\bibitem[Lu]{Lu} N. Luzin: Sur les propri\'et\'es des fonctions mesurables. \em Comptes Rendus Acad. Sci. Paris \em {\bf154} (1912), 1688--1690.

\bibitem[Ma]{Ma} J. Mal\'y: A simple proof of the Stepanov theorem on differentiability almost everywhere. \em Exposition. Math.\em, {\bf17} (1999)59--61.

\bibitem[M]{M} M. Miranda Jr.: Functions of bounded variation on ``good'' metric spaces. \em J. Math. Pures Appl. \em {\bf82} (2003) 975--1004.

\bibitem[R]{R} A. Ranjbar-Motlagh: Generalized Stepanov type theorem with applications over metric-measure spaces. \em Houston J. Math. \em {\bf34} (2008), no. 2, 623--635.

\bibitem[S]{S} S. Saks: Theory of the integral. Monographie Mat., tom7,Warsaw,(1937).

\bibitem[Se]{Se} S. Semmes: Some Novel Types of Fractal Geometry.
Oxford Science Publications (2001).

\bibitem[Sh]{Sh}N. Shanmugalingam: Newtonian Spaces: An
extension of Sobolev spaces to Metric Measure Spaces. \em Rev. Mat.
Iberoamericana\em, {\bf16} (2000), 243--279.

\bibitem[St]{St} W. Stepanoff: Sur les conditions de l'existence de la differentielle totale. \em Rec. Math. Soc. Math. Moscou \em {\bf 32} (1925), 511--526.

\bibitem[Wh]{Wh} H. Whitney: On totally differentiable and smooth functions. \em Pacific J. Math.\em {\bf1} (1951),
143--159.


\end{thebibliography}
\end{document}